\documentclass[11pt]{article}
\usepackage{amsmath,fullpage,amssymb,amsthm,hyperref,enumerate}
\usepackage{tikz,color}
\usepackage{youngtab}
\usetikzlibrary{patterns,shapes,arrows,decorations.pathreplacing}
\tikzstyle{pnt}=[draw,ellipse,fill,inner sep=1pt]

\newtheorem{theorem}{Theorem}[section]
\newtheorem{proposition}[theorem]{Proposition}
\newtheorem{lemma}[theorem]{Lemma}
\newtheorem{corollary}[theorem]{Corollary}

\theoremstyle{definition}
  
\newtheorem{definition}[theorem]{Definition}   

\theoremstyle{remark}

\newcommand{\C}{\mathcal{C}}

\newcommand{\D}{\mathcal{D}}
\renewcommand{\S}{\mathcal{S}}

\newcommand{\si}{\sigma}
\newcommand{\bij}{\phi}

\DeclareMathOperator\des{des}
\DeclareMathOperator\Des{Des}
\DeclareMathOperator\asc{asc}
\DeclareMathOperator\Asc{Asc}
\DeclareMathOperator\Plat{Plat}
\DeclareMathOperator\plat{plat}

\DeclareMathOperator\stat{stat}

\DeclareMathOperator\SYT{SYT}
\newcommand\row{\rho_T}

\newcommand{\poly}{C}

\newcommand{\card}[1]{{\lvert #1 \rvert}}
\newcommand{\delete}[1]{}

\newcommand\nnk{\mathcal{M}_n^k}

\newcommand\uu{\texttt{u}}
\newcommand\dd{\texttt{d}}

\newcommand\hN{\widehat{N}}
\newcommand\blu{\textcolor{blue}}
\newcommand\red{\textcolor{red}}
\newcommand\rd[1]{\mathbf{\textcolor{red}{#1}}}
\newcommand\bl[1]{\mathbf{\textcolor{blue}{#1}}}
\newcommand\vi[1]{\mathbf{\textcolor{violet}{#1}}}
\newcommand\desbar[1]{\draw[brown] #1+(-.3,0)--++(.3,0);}

\newcommand\arc[2]{ \draw[blue,very thick] (#1)  to [bend left=45] (#2);}

\title{Canon permutations and generalized descents\\ of standard Young tableaux}

\author{Sergi Elizalde\thanks{Department of Mathematics, Dartmouth College, Hanover, NH 03755. \texttt{sergi.elizalde@dartmouth.edu}}}

\date{}

\begin{document}

\maketitle

\begin{abstract}
Canon permutations are permutations of the multiset having $k$ copies of each integer between $1$ and $n$, with the property that
the subsequences obtained by taking the $j$th copy of each entry, for each fixed $j$, are all the same.
For $k=2$, canon permutations are sometimes called nonnesting permutations, and it is known that the polynomial that enumerates them by the number of descents factors as a product of an Eulerian polynomial and a Narayana polynomial.
We extend this result to arbitrary $k$, and we relate the problem to the enumeration of standard Young tableaux of rectangular shape with respect to generalized descent statistics. Our proof is bijective, and it also settles a conjecture of Sulanke about the distribution of certain lattice path statistics.
\end{abstract}

\section{Introduction}

\subsection{Canon permutations}

Let $\S_n$ be the set of permutations of $[n]=\{1,2,\dots,n\}$.
Let $\nnk=\{1^k,2^k,\dots,n^k\}$ be the multiset consisting of $k$ copies of each number in $[n]$. Given a permutation $\pi$ of $\nnk$, one can consider the subsequence obtained by taking the first (i.e.\ leftmost) copy of each number in $[n]$, and more generally, the subsequence obtained by taking the $j$th copy from the left of each number in $[n]$, for any given $j\in[k]$. If the subsequences obtained in this way (which must be permutations in $\S_n$)
are the same for every $j$, then $\pi$ is called a {\em canon permutation}. We denote by $\C^k_n$ the set of canon permutations of $\nnk$, and by $\C^{k,\sigma}_n$ the subset of those where the $j$th copies of each entry form the subsequence $\sigma\in\S_n$.
For example,  $351335212514424\in\C^3_5$, since the subsequences of first, second, and third copies of each entry, respectively, are all equal to $\sigma=35124$.

The term {\em canon permutation} was introduced in~\cite{elizalde_descents_2023} in reference to the musical form where different voices play the same melody, starting at different times.

When $k=2$, canon permutations are sometimes called {\em nonnesting permutations}, since they can be interpreted as nonnesting matchings of $[2n]$ where each of the $n$ arcs is labeled with a distinct element in $[n]$. Recall that a perfect matching of $[2n]$ is nonnesting if it does not contain a pair of arcs $(i_1,i_3)$ and $(i_2,i_4)$ where $i_1<i_2<i_3<i_4$.  The corresponding nonnesting permutation is obtained by simply reading, for each $i$ from $1$ to $2n$, the label of the arc containing $i$; see Figure~\ref{fig:matching} for an example. Indeed, the nonnesting condition on the matching is equivalent to the fact that the left endpoints of the arcs occur in the same order as their right endpoints.

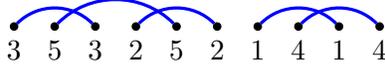
\begin{figure}[htb]
\centering
  \begin{tikzpicture}[scale=0.54]
   \foreach [count=\i] \j in {3,5,3,2,5,2,1,4,1,4}
        \node[pnt,label=below:$\j$] at (\i,0)(\i) {};
   \arc{1}{3} \arc{2}{5} \arc{4}{6} \arc{7}{9} \arc{8}{10}
   \end{tikzpicture}
\caption{The nonnesting permutation $3532521414\in\C^2_5$ viewed as a labeled nonnesting matching.}
\label{fig:matching}
\end{figure}

Nonnesting permutations have recently been considered in~\cite{elizalde_descents_2023}, as a variation of noncrossing permutations~\cite{archer_pattern_2019,elizalde_descents_2021}, which in turn generalize the well-known Stirling permutations introduced by Gessel and Stanley~\cite{gessel_stirling_1978}.
Separately nonnesting permutations naturally index the regions of the Catalan hyperplane arrangement, as shown by Bernardi in~\cite{bernardi_deformations_2018},
where they are called {\em annotated $1$-sketches}. 

In many of the above papers~\cite{gessel_stirling_1978,archer_pattern_2019,elizalde_descents_2021,elizalde_descents_2023}, a statistic that plays an important role is the number of descents. 
We say that $i$ is a {\em descent} of $\pi=\pi_1\pi_2\dots\pi_m$ (a word over the positive integers) if $\pi_i>\pi_{i+1}$, 
and that it is a {\em plateau} if $\pi_i=\pi_{i+1}$.
Denote the {\em descent set} and the {\em plateau set} of $\pi$ by 
\begin{align*}\Des(\pi)&=\{i\in[m-1]:\pi_i>\pi_{i+1}\},\\
\Plat(\pi)&=\{i\in[m-1]:\pi_i=\pi_{i+1}\},
\end{align*}
respectively, its number of descents by $\des(\pi)=\card{\Des(\pi)}$, and its number of plateaus by 
$\plat(\pi)=\card{\Plat(\pi)}$. 

The distribution of descents on $\S_n$ is given by the Eulerian polynomials
\begin{equation}\label{eq:Eulerian_def}
A_n(t)=\sum_{\sigma\in\S_n} t^{\des(\sigma)},
\end{equation} 
whose exponential generating function can be expressed as 
\begin{equation}\label{eq:EulerianGF}\sum_{n\ge0} A_n(t) \frac{z^n}{n!}=\frac{t-1}{t-e^{(t-1)z}};\end{equation}
see \cite[Prop.~1.4.5]{stanley_enumerative_2012}.

In this paper, we are interested in the polynomials
\begin{equation}\label{eq:polydef}
\poly^k_n(t,u)=\sum_{\pi\in\C^k_n} t^{\des(\pi)}u^{\plat(\pi)},
\end{equation}
which give the distribution of the number of descents and plateaus on canon permutations.
 It will be convenient to write $\poly^k_n(t,u)=\sum_{\sigma\in\S_n}\poly^{k,\si}_n(t,u)$, where
$$\poly^{k,\si}_n(t,u)=\sum_{\pi\in\C^{k,\si}_n} t^{\des(\pi)}u^{\plat(\pi)}.$$

The main result in \cite{elizalde_descents_2023} is a formula for these polynomials in the case $k=2$. Before we can state it, we need to define another family of polynomials. Let $\D_n$ be the set of Dyck paths of semilength $n$, that is, lattice paths from $(0,0)$ to $(2n,0)$ with steps $\uu=(1,1)$ and $\dd=(1,-1)$ that do not go below the $x$-axis. A {\em peak} in a Dyck path is an occurrence of two adjacent steps $\uu\dd$; it is called {\em low peak} if these steps touch the $x$-axis, and a {\em high peak} otherwise. Let $\hN_n(t,u)$ be the polynomial enumerating elements in $\D_n$ where $t$ and $u$ mark the number of high peaks and the number of low peaks, respectively. The generating function for these polynomials, called {\em Narayana polynomials}, is 
\begin{equation}\label{eq:Ntu}
\sum_{n\ge0}\hN_n(t,u)z^n=\frac{2}{1+(1+t-2u)z+\sqrt{1-2(1+t)z+(1-t)^2z^2}}.
\end{equation}
It is shown in~\cite{elizalde_descents_2023} that, when $k=2$, the polynomials in equation~\eqref{eq:polydef} factor as follows.

\begin{theorem}[\cite{elizalde_descents_2023}]\label{thm:nonnesting} 
For $n\ge1$,
$$\poly^2_n(t,u)=A_n(t)\,\hN_n(t,u).$$
More specifically, for all $\sigma\in\S_n$, 
$$\poly_n^{2,\sigma}(t,u)=t^{\des(\sigma)}\hN_n(t,u).$$
\end{theorem}

The main goal of this paper is to generalize this result to arbitrary $k$. 
As we will see, in the general version, the role of Dyck paths will be played by standard Young tableaux of rectangular shape.

\section{Generalized descents of standard Young tableaux}

\subsection{Generalized Narayana polynomials}

For $n,k\ge1$, let $\SYT(k^n)$ denote the set of standard Young tableaux of rectangular shape consisting of $n$ rows and $k$ columns. Such tableaux have $kn$ cells, filled with the numbers $1,2,\dots,kn$ so that numbers in rows increase from left to right, and numbers in columns increase from top to bottom. The {\em descent set} of $T\in\SYT(k^n)$, denoted by $\Des(T)$, is the set of $i\in[kn-1]$ such that $i$ appears in a row higher than $i+1$ in $T$. The number of descents of $T$ is $\des(T)=\card{\Des(T)}$.
If we denote by $\row(i)$ the index of the row in $T$ where $i$ appears, where rows are indexed from $1$ to $n$ starting from the top row, then we can write
$$\Des(T)=\{i\in[kn-1]:\row(i)<\row(i+1)\}.$$
Similarly, define the ascent and plateau sets of $T$ by
\begin{align*}\Asc(T)&=\{i\in[kn-1]:\row(i)>\row(i+1)\},\\
\Plat(T)&=\{i\in[kn-1]:\row(i)=\row(i+1)\},
\end{align*}
and let $\asc(T)=\card{\Asc(T)}$ and $\plat(T)=\card{\Plat(T)}$.

The distribution of the number of ascents and descents on standard Young tableaux of rectangular shape is given by the {\em generalized Narayana numbers}, which were studied by Sulanke~\cite{sulanke_generalizing_2004} in connection to higher-dimensional lattice paths, 
and are defined as $$N(n,k,h)=|\{T\in\SYT(k^n):\asc(T)=h\}|.$$
It is also shown in~\cite{sulanke_generalizing_2004} that 
\begin{equation}\label{eq:Ndes} 
N(n,k,h)=|\{T\in\SYT(k^n):\des(T)=h+n-1\}|. 
\end{equation}
Sulanke obtains the following formula for the generalized Narayana numbers, which is implicit in MacMahon's work on plane partitions~\cite{macmahon_combinatory_1960}. We state it in a slightly different form, in order to highlight the symmetry $N(n,k,h)=N(k,n,h)$.

\begin{proposition}[{\cite[Prop.~1]{sulanke_generalizing_2004}}]
For $0\le h\le (k-1)(n-1)$, 
$$N(n,k,h)=\sum_{\ell=0}^h (-1)^{h-\ell}\binom{kn+1}{h-\ell}\prod_{i=0}^{n-1}\prod_{j=0}^{k-1}\frac{i+j+1+\ell}{i+j+1}.$$
\end{proposition}

Define now the {\em generalized Narayana polynomials}
\begin{equation}\label{eq:Nnk}
N_{n,k}(t,u)=\sum_{T\in\SYT(k^n)}t^{\asc(T)} u^{\plat(T)},
\end{equation}
and note that 
$$N_{n,k}(t,1)=\sum_{h=0}^{(k-1)(n-1)}N(n,k,h)\,t^h.$$

For $k=2$, there is a straightforward bijection between $\SYT(2^n)$ and the set $\D_n$ of Dyck paths, obtained by letting the $i$th step of the path be $\uu$ if $i$ is in the first column of the tableau, and $\dd$ otherwise, for each $i\in[2n]$. Under this bijection, ascents of the tableau become high peaks of the Dyck path, and plataus become low peaks. It follows that $N_{n,2}(t,u)=\hN_n(t,u)$, the polynomials from equation~\eqref{eq:Ntu}.

\subsection{Generalized descents}

Next we generalize the definitions of descents and ascents in standard Young tableaux. 

\begin{definition}\label{def:Des_sigma}
For any permutation $\sigma=\sigma_1\dots\sigma_n\in\S_n$, define the {\em $\sigma$-descent set} of $T\in\SYT(k^n)$ by 
$$\Des_\sigma(T)=\{i\in[kn-1]:\sigma_{\row(i)}>\sigma_{\row(i+1)}\},$$
and the number of $\sigma$-descents by $\des_\sigma(T)=\card{\Des_\sigma(T)}$.
\end{definition}

For example, is $\sigma=35142\in\S_5$ and $T\in\SYT(3^5)$ is the tableau on the left of Figure~\ref{fig:Des_sigma}, we have $\Des_\sigma(T)=\{2,4,7,9,10,11,14\}$ and $\des_\sigma(T)=7$.
Definition~\ref{def:Des_sigma} generalizes the usual descent and ascent sets of standard Young tableaux, since 
$$\Des(T)=\Des_{n\dots21}(T) \quad\text{and}\quad \Asc(T)=\Des_{12\dots n}(T)$$
for all $T\in\SYT(k^n)$.

\begin{figure}[htb]
\centering
 \begin{tikzpicture}[scale=.5]
\draw (0,0) grid (3,5);
\node at (.5,4.5) {$1$}; \node at (1.5,4.5) {$3$}; \node at (2.5,4.5) {$6$};
\node at (.5,3.5) {$2$}; \node at (1.5,3.5) {$4$}; \node at (2.5,3.5) {$9$};
\node at (.5,2.5) {$5$}; \node at (1.5,2.5) {$8$}; \node at (2.5,2.5) {$12$};
\node at (.5,1.5) {$7$}; \node at (1.5,1.5) {$10$}; \node at (2.5,1.5) {$14$};
\node at (.5,.5) {$11$}; \node at (1.5,.5) {$13$}; \node at (2.5,.5) {$15$};
\node[brown] at (-.5,5.5) {$\sigma$}; \node[brown,scale=.8] at (-.5,4.5) {$3$}; \node[brown,scale=.8] at (-.5,3.5) {$5$}; \node[brown,scale=.8] at (-.5,2.5) {$1$}; \node[brown,scale=.8] at (-.5,1.5) {$4$}; \node[brown,scale=.8] at (-.5,.5) {$2$};
\node at (4.25,2.5) {$\mapsto$};
\node[right] at (5,2.5) {$353513415421242=\pi$
};
 \end{tikzpicture}
 \caption{The bijection from Lemma~\ref{lem:sigmaT}, with $\sigma=35142$. 
 Note that $\Des_\sigma(T)=\Des(\pi)=\{2,4,7,9,10,11,14\}$.}
\label{fig:Des_sigma}
\end{figure}

Standard Young tableaux of rectangular shape are related to canon permutations via the following bijection; see Figure~\ref{fig:Des_sigma} for an example. 

\begin{lemma}\label{lem:sigmaT}
For every $\sigma\in\S_n$, the map 
$$\begin{array}{ccc}
\SYT(k^n)&\longrightarrow&\C^{k,\sigma}_n\\
T&\mapsto&\pi,
\end{array}
$$
where $\pi=\sigma_{\row(1)}\sigma_{\row(2)}\dots \sigma_{\row(kn)}$,
is a bijection. Additionally, $$\Des_\sigma(T)=\Des(\pi) \quad\text{and}\quad \Plat(T)=\Plat(\pi),$$
which implies that 
$$\sum_{T\in\SYT(k^n)} t^{\des_\sigma(T)} u^{\plat(T)}=\poly_n^{k,\sigma}(t,u).$$
\end{lemma}

\begin{proof}
Let us show that, for fixed $\sigma\in\S_n$, the map $T\mapsto\pi$ is a bijection between $\SYT(k^n)$ and $\C^{k,\sigma}_n$.
By construction, entry $i$ is in row $r$ and column $j$ of $T$ if and only if $\pi_i$ is the $j$th copy (from the left) of $\sigma_r$ in $\pi$. Thus, the entries in column $j$ of $T$ are the positions in $\pi$ of the $j$th copies of each number, which, when read from left to right, form the permutation $\sigma$. This proves that $\pi\in\C^{k,\sigma}_n$. Conversely, given $\pi\in\C^{k,\sigma}_n$, we can recover $T$ using the above observation. 

Finally, note that $i\in\Des(\pi)$ if and only if $\pi_i>\pi_{i+1}$, that is, $\sigma_{\row(i)}>\sigma_{\row(i+1)}$, which is equivalent to $i\in\Des_\sigma(T)$. Similarly, $i\in\Plat(\pi)$ if and only if $\pi_i=\pi_{i+1}$, that is, $\sigma_{\row(i)}=\sigma_{\row(i+1)}$, which is equivalent to $\row(i)=\row(i+1)$, and also to $i\in\Plat(T)$.
\end{proof}

By letting $\sigma\in\S_n$ in Lemma~\ref{lem:sigmaT} vary, we obtain a bijection
$$\begin{array}{ccc}
\S_n\times\SYT(k^n)&\to&\C^k_n\\
(\sigma,T)&\mapsto&\pi.
\end{array}
$$

\section{Main result}

Our main result is the following generalization of Theorem~\ref{thm:nonnesting} to arbitrary $k$.

\begin{theorem}\label{thm:main} 
For $n,k\ge1$,
$$\poly^k_n(t,u)=A_n(t)\,N_{n,k}(t,u).$$
More specifically, for all $\sigma\in\S_n$, 
\begin{equation}\label{eq:polysigma}
\poly_n^{k,\sigma}(t,u)=t^{\des(\sigma)}N_{n,k}(t,u).
\end{equation}
\end{theorem}

The second statement immediately implies the first by summing over all $\sigma\in\S_n$. 
Using Lemma~\ref{lem:sigmaT} and equation~\eqref{eq:Nnk}, we can rewrite equation~\eqref{eq:polysigma} as
$$\sum_{T\in\SYT(k^n)} t^{\des_\sigma(T)} u^{\plat(T)}=t^{\des(\sigma)}\sum_{T\in\SYT(k^n)}t^{\asc(T)} u^{\plat(T)}.$$
Note that $\asc(T)=\des_{12\dots n}(T)$. Thus, in order to prove Theorem~\ref{thm:main}, it suffices to construct a bijection
$\bij_\sigma:\SYT(k^n)\to\SYT(k^n)$ such that, for all $T\in\SYT(k^n)$,
\begin{equation}\label{eq:bij_property}
\des_\sigma(T)=\des_{12\dots n}(\bij_\sigma(T))+\des(\sigma)\quad\text{and}\quad\plat(T)=\plat(\bij_\sigma(T)).
\end{equation}
We will do this in Section~\ref{sec:bij}.

In the special case that $\sigma=n\dots21$, using that $\des_{n\dots21}(T)=\des(T)$ and $\des(n\dots21)=n-1$, the bijection $\bij_{n\dots21}$ has the property that
$$\des(T)=\asc(\bij_{n\dots21}(T))+n-1.$$
A similar bijection with this property was given by Sulanke in~\cite[Prop.~2]{sulanke_generalizing_2004}, in order to prove equation~\eqref{eq:Ndes}. 

In the same paper, in a remark at the end of \cite[Sec.~3.1]{sulanke_generalizing_2004}, Sulanke defines certain statistics $\asc_\tau$ on higher-dimensional lattice paths, which are closely related to our statistics $\des_\sigma$, and he conjectures that they have a (shifted) generalized Narayana distribution. Sulanke's conjecture can be shown to be equivalent to our Theorem~\ref{thm:main} when $u=1$.

\section{The bijections}\label{sec:bij}

The goal of this section is to construct a bijection $\bij_\sigma$ that satisfies equation~\eqref{eq:bij_property}, relating the statistics $\des_\sigma$ and $\des_{12\dots n}$, while preserving the number of plateaus. This will be achieved in two steps. In Section~\ref{sec:nonadjacent}, we describe a sequence of bijections that relates the statistics $\des_\sigma$ and $\des_\lambda$, where $\lambda$ is a very specific permutation having the same descent set as $\sigma$. In Section~\ref{sec:removing}, we describe a sequence of bijections that relates the statistics $\des_\lambda$ and $\des_{12\dots n}$. The bijection $\bij_\sigma$ will be obtained by composing the two sequences of bijections.

It will be convenient to identify $T\in\SYT(k^n)$ with its {\em tableau word} $w_T=\row(1)\row(2)\dots\row(kn)$. This encoding is simply the map from Lemma~\ref{lem:sigmaT} when $\sigma$ is the identity permutation, and so 
$w_T\in\C^{k,12\dots n}_n$. As an alternative characterization, a permutation of $\nnk$ is the tableau word of some $T\in\SYT(k^n)$ if and only if, in every prefix, the number of copies of $j$ is greater than or equal to the number of copies of $j+1$, for all $j\in[k-1]$. We call this the {\em ballot condition}.

We say that two statistics $\stat_1$ and $\stat_2$ on a set $X$ are equidistributed if there is a bijection $f:X\to X$ such that, for all $x\in X$, we have $\stat_1(x)=\stat_2(f(x))$. We allow $\stat_1$ and $\stat_2$ to be tuples of statistics, which can be integer-valued or set-valued.

\subsection{Switching non-adjacent entries of $\sigma$}\label{sec:nonadjacent}

Let $r,s\in[n]$ be such that $|r-s|>1$. We start by defining two bijections $f_{rs},F_{rs}:\SYT(k^n)\to\SYT(k^n)$ that affect only the entries in rows $r$ and $s$ of the tableau. Given $T\in\SYT(k^n)$, we will describe $f_{rs}(T)$ and $F_{rs}(T)$ in terms of their tableau words, which will be obtained from $w_T$ by applying certain permutations to each of the maximal consecutive blocks $B$ having entries in $\{r,s\}$.
Since rows $r$ and $s$ are not adjacent, permuting the entries within each such block $B$ does not violate the ballot condition,
and so the resulting word is guaranteed to encode a tableau in $\SYT(k^n)$.

\begin{definition}\label{def:frs}
To construct the tableau word of $f_{rs}(T)$, we change each block $B$ to $B'$ as follows. 
\begin{enumerate}[(1)]
\item If $B$ starts and ends with the same letter, let $B'=B$.
\item Otherwise, if $B=r^{a_1}s^{b_1}\dots r^{a_j}s^{b_j}$ for some $j\ge1$ and $a_i,b_i\ge1$ for all $i$, let $B'=s^{b_1}r^{a_1}\dots s^{b_j}r^{a_j}$;
conversely, if $B=s^{b_1}r^{a_1}\dots s^{b_j}r^{a_j}$, let  $B'=r^{a_1}s^{b_1}\dots r^{a_j}s^{b_j}$.
\end{enumerate}
\end{definition}
For example, if $B=srrsssrr$, we get $B'=rrsrrsss$. 

\begin{figure}[htb]
\centering
 \begin{tikzpicture}[scale=.5]
 \fill[violet!35] (0,1) rectangle (6,2);
  \fill[violet!35] (0,3) rectangle (6,4);
\draw (0,0) grid (6,4);
\node at (.5,3.5) {$1$}; \node at (1.5,3.5) {$3$}; \node at (2.5,3.5) {$5$}; \node at (3.5,3.5) {$9$}; \node at (4.5,3.5) {$11$}; \node at (5.5,3.5) {$12$};
\node at (.5,2.5) {$2$}; \node at (1.5,2.5) {$7$}; \node at (2.5,2.5) {$8$};  \node at (3.5,2.5) {$14$}; \node at (4.5,2.5) {$16$}; \node at (5.5,2.5) {$20$};
\node at (.5,1.5) {$4$}; \node at (1.5,1.5) {$10$}; \node at (2.5,1.5) {$13$};  \node at (3.5,1.5) {$17$}; \node at (4.5,1.5) {$18$}; \node at (5.5,1.5) {$22$};
\node at (.5,.5) {$6$}; \node at (1.5,.5) {$15$}; \node at (2.5,.5) {$19$};  \node at (3.5,.5) {$21$}; \node at (4.5,.5) {$23$}; \node at (5.5,.5) {$24$};
\node[left] at (0,2) {$T=$};
\node at (2,-1)  {$w_T=12\hspace{.1em}\vi{\underline{131}}\hspace{.1em}422\hspace{.1em}\vi{\underline{13113}}\hspace{.1em}242\vi{33}424\vi{3}44$};
\draw[->] (8,3)--node[above] {$f_{13}$} (11,5);
\draw[->] (8,.5)--node[above] {$F_{13}$} (11,-1.5);

\begin{scope}[shift={(12.5,3.5)}]
 \fill[violet!35] (0,1) rectangle (6,2);
  \fill[violet!35] (0,3) rectangle (6,4);
\draw (0,0) grid (6,4);
\node at (.5,3.5) {$1$}; \node at (1.5,3.5) {$3$}; \node at (2.5,3.5) {$5$}; \node at (3.5,3.5) {$10$}; \node at (4.5,3.5) {$12$}; \node at (5.5,3.5) {$13$};
\node at (.5,2.5) {$2$}; \node at (1.5,2.5) {$7$}; \node at (2.5,2.5) {$8$};  \node at (3.5,2.5) {$14$}; \node at (4.5,2.5) {$16$}; \node at (5.5,2.5) {$20$};
\node at (.5,1.5) {$4$}; \node at (1.5,1.5) {$9$}; \node at (2.5,1.5) {$11$};  \node at (3.5,1.5) {$17$}; \node at (4.5,1.5) {$18$}; \node at (5.5,1.5) {$22$};
\node at (.5,.5) {$6$}; \node at (1.5,.5) {$15$}; \node at (2.5,.5) {$19$};  \node at (3.5,.5) {$21$}; \node at (4.5,.5) {$23$}; \node at (5.5,.5) {$24$};
\node at (3,-1) {$12\hspace{.1em}\vi{\underline{131}}\hspace{.1em}422\hspace{.1em}\vi{\underline{31311}}\hspace{.1em}242\vi{33}424\vi{3}44$};
\end{scope}

\begin{scope}[shift={(12.5,-3.5)}]
 \fill[violet!35] (0,1) rectangle (6,2);
  \fill[violet!35] (0,3) rectangle (6,4);
\draw (0,0) grid (6,4);
\node at (.5,3.5) {$1$}; \node at (1.5,3.5) {$3$}; \node at (2.5,3.5) {$4$}; \node at (3.5,3.5) {$10$}; \node at (4.5,3.5) {$12$}; \node at (5.5,3.5) {$13$};
\node at (.5,2.5) {$2$}; \node at (1.5,2.5) {$7$}; \node at (2.5,2.5) {$8$};  \node at (3.5,2.5) {$14$}; \node at (4.5,2.5) {$16$}; \node at (5.5,2.5) {$20$};
\node at (.5,1.5) {$5$}; \node at (1.5,1.5) {$9$}; \node at (2.5,1.5) {$11$};  \node at (3.5,1.5) {$17$}; \node at (4.5,1.5) {$18$}; \node at (5.5,1.5) {$22$};
\node at (.5,.5) {$6$}; \node at (1.5,.5) {$15$}; \node at (2.5,.5) {$19$};  \node at (3.5,.5) {$21$}; \node at (4.5,.5) {$23$}; \node at (5.5,.5) {$24$};
\node at (3,-1) {$12\hspace{.1em}\vi{\underline{113}}\hspace{.1em}422\hspace{.1em}\vi{\underline{31311}}\hspace{.1em}242\vi{33}424\vi{3}44$};
\end{scope}
 \end{tikzpicture}
 \caption{The bijections $f_{rs}$ and $F_{rs}$ for $r=1$ and $s=3$. The blocks $B,B',B''$ containing two different letters and underlined.
Letting $\sigma=2431$ and $\tau=3421$, we have $\des_\sigma(T)=10=\des_\tau(f_{13}(T))$, $\plat(T)=4=\plat(f_{13}(T))$, and 
$\Des_\sigma(T)=\{2,4,5,8,10,14,16,18,20,22\}=\Des_\tau(F_{13}(T))$.
}
\label{fig:frs}
\end{figure}
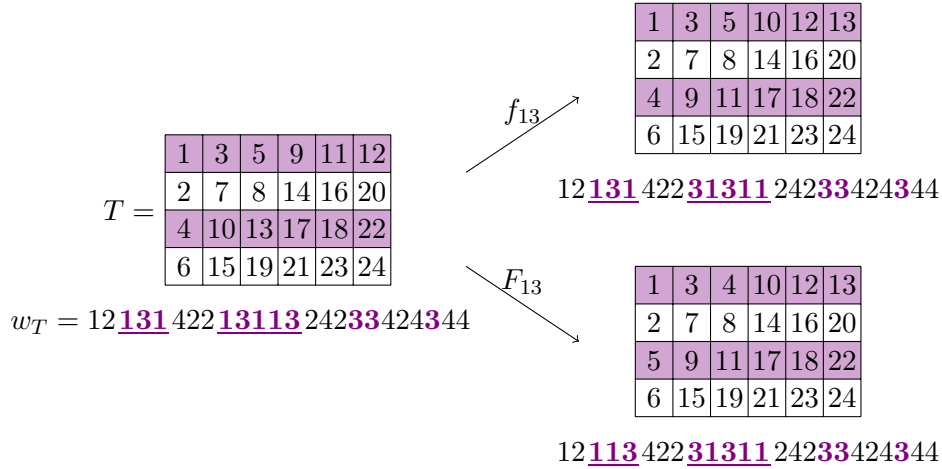

\begin{definition}\label{def:Frs}
To construct  the tableau word of $F_{rs}(T)$, we change each block $B$ to $B''$ as follows.
\begin{enumerate}[(1)]
\item Break $B$ into subblocks by splitting at each $s$ followed by $r$. For this, we view $B$ as a cyclic word so that its last entry is immediately followed by its first entry. In particular, unless $B$ ends with an $s$ and starts with an $r$, one of the subblocks straddles from the end to the beginning of $B$. 
\item For each subblock, write it as $r^as^b$ for some $a,b\ge1$, and change it to $s^br^a$, while keeping the splitting points between subblocks. Let $B''$ be the resulting block.
\end{enumerate}
\end{definition}
For example, if $B=srrsssrr$, we split it into two sublocks as $B=s{/}rrsss{/}rr$ (with the block $rrs$ stradling across the endpoints), and then change it into $B''=r{/}sssrr{/}sr$. See Figure~\ref{fig:frs} for a complete example of $f_{rs}$ and $F_{rs}$ applied to a tableau. 

It is clear from the definitions that $f_{rs}^{-1}=f_{sr}=f_{rs}$, so $f_{rs}$ is an involution, and that $F_{rs}^{-1}=F_{sr}$. Next we look at the effect of these bijections on descents and plateaus.

\begin{lemma}\label{lem:fF}
Let $\sigma\in\S_n$, let $r,s\in[n]$ be such that $|r-s|>1$, and suppose that $\sigma_r=\ell$ and $\sigma_s=\ell+1$.
Let $\tau\in\S_n$ be  obtained from $\sigma$ by switching these two entries, that is, $\tau_r=\ell+1$, $\tau_s=\ell$, and $\tau_i=\sigma_i$ for $i\notin\{r,s\}$. Then the bijections $f_{rs},F_{rs}:\SYT(k^n)\to\SYT(k^n)$ satisfy that, for all $T\in\SYT(k^n)$,
$$\des_\sigma(T)=\des_\tau(f_{rs}(T)) , \quad \plat(T)=\plat(f_{rs}(T)),$$
$$\Des_\sigma(T)=\Des_\tau(F_{rs}(T)).$$
\end{lemma}

\begin{proof}
Let us first consider the effect of $f_{rs}$ on the generalized descents and plateaus of the tableau. Clearly, $\plat(T)=\plat(f_{rs}(T))$, since plateaus correspond to pairs of adjacent equal entries in the tableau word, and the number of these is not affected by $f_{rs}$. 

As for descents, recall that $i\in\Des_\sigma(T)$ if $\sigma_{\row(i)}>\sigma_{\row(i+1)}$. Since the values $\ell$ and $\ell+1$ are consecutive, the only elements where $\Des_\sigma(T)$ and $\Des_\tau(f_{rs}(T))$ may differ must come from positions inside the blocks $B$ and $B'$ in Definition~\ref{def:frs}: adjacent pairs $sr$ in $B$ contribute to $\Des_\sigma(T)$, whereas adjacent pairs $rs$ in $B'$ contribute to $\Des_\tau(f_{rs}(T))$.  
In case (1), the block $B=B'$ has the same number of adjacent pairs $sr$ as adjacent pairs $rs$. In case (2), the number of adjacent pairs $sr$ in $B$ is equal to the number of adjacent pairs $rs$ in $B'$. It follows that $\des_\sigma(T)=\des_\tau(f_{rs}(T))$.

Let us now consider the effect of $F_{rs}$ on the generalized descents.
The only elements where $\Des_\sigma(T)$ and $\Des_\tau(F_{rs}(T))$ may differ must come from positions inside the blocks $B$ and $B''$ in Definition~\ref{def:Frs}: adjacent pairs $sr$ in $B$ contribute to $\Des_\sigma(T)$, whereas adjacent pairs $rs$ in $B''$ contribute to $\Des_\sigma(F_{rs}(T))$. The locations of these adjacent pairs are precisely the splitting points for the subblocks (not including a potential splitting point at the very end of the block), so they are preserved. We conclude that $\Des_\sigma(T)=\Des_\tau(F_{rs}(T))$.
\end{proof}

When $k=2$, the blocks $B$ can have size at most two, and the bijections $f_{rs}$ and $F_{rs}$ coincide, sending $\Des_\sigma$ to $\Des_\tau$ while simultaneously preserving $\plat$. However, for $k\ge3$, the pairs of statistics $(\Des_\sigma,\plat)$ and $(\Des_\tau,\plat)$ are not equidistributed on $\SYT(k^n)$ in general.

Recall that a permutation is {\em reverse-layered} if it can be decomposed as a sequence of monotone increasing blocks where the entries in each block are larger than the entries in the next block. Each block in this decomposition is called a layer. For example, $789{|}6{|}345{|}12\in\S_9$ is reverse-layered, with layers of size $3,1,3,2$, separated by vertical bars. For any subset $S\subseteq[n-1]$, there is a unique reverse-layered permutation $\lambda\in\S_n$ with $\Des(\lambda)=S$. If the elements of $S$ are $i_1<i_2<\dots<i_d$, then
$$
    \lambda=(n{-}i_1{+}1)(n{-}i_1{+}2)\ldots n{|}(n{-}i_2{+}1)(n{-}i_2{+}2) \ldots (n{-}i_1){|}\dots{|}12{\dots} (n{-}i_d).
$$

Given any two permutations in $\S_n$ with the same descent set, we can transform one into the other by repeatedly switching non-adjacent entries with consecutive values. Indeed, suppose that $\sigma\in\S_n$ is not reverse-layered, and that $\lambda\in\S_n$ is the unique reverse-layered permutation such that $\Des(\sigma)=\Des(\lambda)$. Note that $\lambda$ has the maximum number of inversions among all the permutations with the same descent set. We can then find a pair of non-adjacent entries in $\sigma$ with consecutive values that do not create an inversion, that is, indices $r,s$ such that $s-r>1$ and $\sigma_s=\sigma_r+1$. Transposing the entries in positions $r,s$ increases the number of inversions of the permutation, while preserving the descent set. 
For concreteness, we consider the specific choice of $r,s$ described in~\cite{elizalde_descents_2023}: let $m$ be the largest entry that is not in the same position in $\sigma$ as in $\lambda$, let $r$ be such that $\lambda_r=m$, and let $s$ be such that $\sigma_s=\sigma_r+1$.
By repeatedly applying such transpositions, one eventually reaches the permutation $\lambda$. For example, if $\sigma=14235$, then $\lambda=45123$, and the sequence of transpositions is
$$14235 \overset{2,5}{\longrightarrow} 15234 \overset{1,3}{\longrightarrow}  25134 \overset{1,4}{\longrightarrow} 35124 \overset{1,5}{\longrightarrow}  45123.$$

Denote by $f_\sigma:\SYT(k^n)\to\SYT(k^n)$ the composition of the bijections $f_{rs}$ corresponding to the above sequence of transpositions that transform $\sigma$ into $\lambda$, and define $F_\sigma$ similarly. In the above example, $f_{14235}=f_{15}\circ f_{14}\circ f_{13}\circ f_{25}$.
Lemma~\ref{lem:fF} implies that, for all $T\in\SYT(k^n)$,
\begin{equation}\label{eq:f-des}
 \des_\sigma(T)=\des_{\lambda}(f_\sigma(T)) , \quad \plat(T)=\plat(f_\sigma(T)),
\end{equation}
$$\Des_\sigma(T)=\Des_{\lambda}(F_\sigma(T)).$$
The compositions $f_\tau^{-1}\circ f_\sigma$ and $F_\tau^{-1}\circ F_\sigma$ give a bijective proof of the following.

\begin{proposition}\label{prop:equidist}
 If $\sigma,\tau\in\S_n$ are such that $\Des(\sigma)=\Des(\tau)$, then
\begin{enumerate}[(i)]
\item the pairs of statistics $(\des_\sigma,\plat)$ and $(\des_\tau,\plat)$ are equidistributed on $\SYT(k^n)$,
\item the statistics $\Des_\sigma$ and $\Des_\tau$ are equidistributed on $\SYT(k^n)$.
\end{enumerate}
\end{proposition}

\subsection{Removing descents}\label{sec:removing}

The previous section allows us to focus on the statistics $\des_\lambda$ where $\lambda\in\S_n$ is a reverse-layered permutation. Recall that, for each $S\subseteq[n-1]$, there is one such permutation having $S$ as its descent set. In this section we analyze how these statistics change when removing elements from $S$ one at a time, with the goal of relating $\des_\lambda$ with $\des_{12\dots n}$.

Let $0\le\ell<m\le n$. Next we define a bijection $g_{\ell m}:\SYT(k^n)\to\SYT(k^n)$, which will be used to analyze the removal of largest element from $S$. Given $T\in\SYT(k^n)$, we will describe $g_{\ell m}(T)$ in terms of its tableau word, similarly to the bijections in Section~\ref{sec:nonadjacent}. Each entry in $w_T$ must belong to one of the sets $X=\{1,\dots,\ell\}$ (which is empty if $\ell=0$), $Y=\{\ell+1,\dots,m\}$, or $Z=\{m+1,\dots,n\}$.  The three subsequences of $w_T$ obtained by restricting to each one of these sets
are not changed by $g_{\ell m}$; rather, $g_{\ell m}$ only affects the interleaving among these subsequences. Thus, to describe $g_{\ell m}$, it will be enough to explain how it changes the simplified tableau word $\overline{w_T}$, defined as the word over $\{x,y,z\}$ whose $i$th entry records which of the sets $X,Y,Z$ the $i$th entry of $w_T$ belongs to.

\begin{definition}\label{def:glm}
To construct the simplified tableau word of $g_{\ell m}(T)$, we consider each maximal block $B$ of $\overline{w_T}$ having entries in $\{x,z\}$, and we change it to $B'$ as follows.
\begin{enumerate}[(1)]
\item If $B$ starts and ends with the same letter, let $B'=B$.
\item Otherwise, if $B=x^{a_1}z^{b_1}\dots x^{a_j}z^{b_j}$ for some $j\ge1$ and $a_i,b_i\ge1$ for all $i$, let $B'=z^{b_1}x^{a_1}\dots z^{b_j}x^{a_j}$;
conversely, if $B=z^{b_1}x^{a_1}\dots z^{b_j}x^{a_j}$, let  $B'=x^{a_1}z^{b_1}\dots x^{a_j}z^{b_j}$.
\end{enumerate}
\end{definition}

See Figure~\ref{fig:glm} for a complete example of $g_{\ell m}$ applied to a tableau. By definition, $g_{\ell m}$ is an involution. Note that, in the special case $\ell=0$, the word $\overline{w_T}$ has no $x$s, and so $g_{0m}$ is simply the identity map. Next we look at the effect of $g_{\ell m}$ on descents and plateaus.

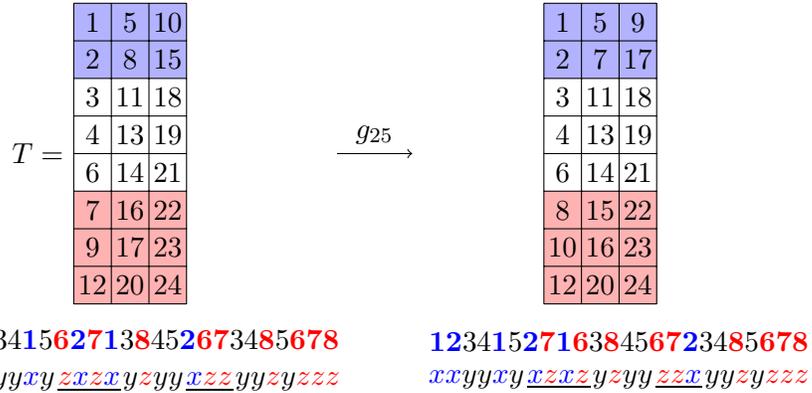
\begin{figure}[htb]
\centering
 \begin{tikzpicture}[scale=.5]
 \fill[red!30] (0,0) rectangle (3,3);
  \fill[blue!30] (0,6) rectangle (3,8);
\draw (0,0) grid (3,8);
\node at (.5,7.5) {$1$}; \node at (1.5,7.5) {$5$}; \node at (2.5,7.5) {$10$}; 
\node at (.5,6.5) {$2$}; \node at (1.5,6.5) {$8$}; \node at (2.5,6.5) {$15$}; 
\node at (.5,5.5) {$3$}; \node at (1.5,5.5) {$11$}; \node at (2.5,5.5) {$18$}; 
\node at (.5,4.5) {$4$}; \node at (1.5,4.5) {$13$}; \node at (2.5,4.5) {$19$}; 
\node at (.5,3.5) {$6$}; \node at (1.5,3.5) {$14$}; \node at (2.5,3.5) {$21$}; 
\node at (.5,2.5) {$7$}; \node at (1.5,2.5) {$16$}; \node at (2.5,2.5) {$22$};  
\node at (.5,1.5) {$9$}; \node at (1.5,1.5) {$17$}; \node at (2.5,1.5) {$23$};  
\node at (.5,.5) {$12$}; \node at (1.5,.5) {$20$}; \node at (2.5,.5) {$24$};  
\node[left] at (0,4) {$T=$};
\node at (1,-1)  {$w_T=\bl{12}34\bl{1}5\rd{6}\bl{2}\rd{7}\bl{1}3\rd{8}45\bl{2}\rd{67}34\rd{8}5\rd{678}$};
\node at (1,-2)  {$\overline{w_T}=\blu{xx}yy\blu{x}y\hspace{.1em}\underline{\red{z}\blu{x}\red{z}\blu{x}}\hspace{.1em}y\red{z}yy\hspace{.1em}\underline{\blu{x}\red{zz}}\hspace{.1em}yy\red{z}y\red{zzz}$};
\draw[->] (7,4)--node[above] {$g_{25}$} (9,4);

\begin{scope}[shift={(12.5,0)}]
 \fill[red!30] (0,0) rectangle (3,3);
  \fill[blue!30] (0,6) rectangle (3,8);
\draw (0,0) grid (3,8);
\node at (.5,7.5) {$1$}; \node at (1.5,7.5) {$5$}; \node at (2.5,7.5) {$9$}; 
\node at (.5,6.5) {$2$}; \node at (1.5,6.5) {$7$}; \node at (2.5,6.5) {$17$}; 
\node at (.5,5.5) {$3$}; \node at (1.5,5.5) {$11$}; \node at (2.5,5.5) {$18$}; 
\node at (.5,4.5) {$4$}; \node at (1.5,4.5) {$13$}; \node at (2.5,4.5) {$19$}; 
\node at (.5,3.5) {$6$}; \node at (1.5,3.5) {$14$}; \node at (2.5,3.5) {$21$}; 
\node at (.5,2.5) {$8$}; \node at (1.5,2.5) {$15$}; \node at (2.5,2.5) {$22$};  
\node at (.5,1.5) {$10$}; \node at (1.5,1.5) {$16$}; \node at (2.5,1.5) {$23$};  
\node at (.5,.5) {$12$}; \node at (1.5,.5) {$20$}; \node at (2.5,.5) {$24$};  
\node at (2,-1)  {$\bl{12}34\bl{1}5\bl{2}\rd{7}\bl{1}\rd{6}3\rd{8}45\rd{67}\bl{2}34\rd{8}5\rd{678}$};
\node at (2,-2)  {$\blu{xx}yy\blu{x}y\hspace{.1em}\underline{\blu{x}\red{z}\blu{x}\red{z}}\hspace{.1em}y\red{z}yy\hspace{.1em}\underline{\red{zz}\blu{x}}\hspace{.1em}yy\red{z}y\red{zzz}$};
\end{scope}
 \end{tikzpicture}
 \caption{The bijection $g_{\ell m}$ for $\ell=2$ and $m=5$. The blocks $B$ and $B'$ in case (2) of Definition~\ref{def:glm} are underlined.
Letting $S=\{2,5\}$, $\lambda=78456123$ and $\lambda'=78123456$, we have $\des_\lambda(T)=9=\des_{\lambda'}(g_{25}(T))+1$ and $\plat(T)=\plat(g_{25}(T))=0$.}
\label{fig:glm}
\end{figure}

\begin{lemma}\label{lem:g}
Let $S\subseteq[n-1]$ be nonempty, let $m=\max S$, and let $S'=S\setminus\{m\}$. Let $\ell=\max S'$ if $S'$ is nonempty, otherwise let $\ell=0$. Let $\lambda,\lambda'\in\S_n$ be the reverse-layered permutations with descent sets $S$ and $S'$, respectively. 
Then the bijection $g_{\ell m}:\SYT(k^n)\to\SYT(k^n)$ satisfies that, for all $T\in\SYT(k^n)$,
$$\des_\lambda(T)=\des_{\lambda'}(g_{\ell m}(T))+1, \qquad \plat(T)=\plat(g_{\ell m}(T)).$$
\end{lemma}

\begin{proof}
It is clear that the number of plateaus is preserved by $g_{\ell m}$, since the number of pairs of equal adjacent letters in the tableau words is unchanged.

The reverse-layered permutations $\lambda$ and $\lambda'$ only differ in the relative order of the entries in positions $Y$ and $Z$: in $\lambda$, entries in positions $Y$ are larger than those in positions $Z$, whereas the opposite is true in $\lambda'$. The relative order of entries in positions within the same set $X$, $Y$ or $Z$ does not change, and neither does the fact that entries in positions $X$ are larger than those in positions $Y$ and $Z$.
It follows that the sets $\Des_\lambda(T)$ and $\Des_{\lambda'}(T)$ only differ in those $i$ such that $\row(i)\in Y$ and  $\row(i+1)\in Z$, or viceversa.
In terms of the simplified tableau word $\overline{w_T}$, an adjacent pair $yz$ contributes to $\Des_\lambda(T)$ but not to $\Des_{\lambda'}(T)$, whereas an adjacent pair $zy$ contributes to $\Des_{\lambda'}(T)$ but not to $\Des_{\lambda}(T)$.

Thus, the only elements where $\Des_\lambda(T)$ and $\Des_{\lambda'}(g_{\ell m}(T))$ may differ must come from positions inside or on the boundary of the blocks $B$ and $B'$ in Definition~\ref{def:glm}. Note that the word $\overline{w_T}$ must start with a nonempty sequence of $x$s before the first $y$ appears, and it must end with a nonempty sequence of $z$s following the last $y$. In particular, all the blocks $B$ that contain both $x$s and $z$s are preceded and followed by a $y$.

In case~(1), we have $B=B'$, so any discrepancies must come from adjacent pairs $yz$ and $zy$ on the boundaries of this block. These can only occur when the block starts and ends in $z$. In this case, the adjacent pair $yz$ on the left boundary only contributes to $\des_\lambda(T)$, whereas the adjacent pair $zy$ on the right boundary only contributes to $\des_{\lambda'}(g_{\ell m}(T))$, so the total contributions are the same. There is, however, one important exception: the block of $z$s at the end of $\overline{w_T}$ is not followed by a $y$, so the contribution of this block to $\des_{\lambda'}(g_{\ell m}(T))$ is one less than its contribution to $\des_\lambda(T)$.

In case~(2), the contributions of $B=x^{a_1}z^{b_1}\dots x^{a_j}z^{b_j}$ to $\des_\lambda(T)$ are the $j$ adjacent pairs $xz$,
whereas the contributions of $B'=z^{b_1}x^{a_1}\dots z^{b_j}x^{a_j}$ to $\des_{\lambda'}(g_{\ell m}(T))$ are the $j-1$ pairs $xz$ and the pair $xy$ on the right boundary, for a total of $j$ in both instances. Similarly, the contributions of $B=z^{b_1}x^{a_1}\dots z^{b_j}x^{a_j}$ to $\des_\lambda(T)$ are the $j-1$ pairs $xz$, the pair $yz$ on the left boundary, and the pair $xy$ on the right boundary, whereas the contributions of $B'=x^{a_1}z^{b_1}\dots x^{a_j}z^{b_j}$ to $\des_{\lambda'}(g_{\ell m}(T))$ are the $j$ pairs $xz$ and the pair $zy$ on the right boundary, for a total of $j+1$ in both instances.

Adding the contributions of all the blocks, it follows that $\des_\lambda(T)=\des_{\lambda'}(g_{\ell m}(T))+1$.
\end{proof}

By repeatedly applying Lemma~\ref{lem:g}, we can construct a bijection that relates the statistics $\des_\lambda$ and $\des_{12\dots n}$. Specifically, if the elements of $S\subseteq[n-1]$ are $i_1<i_2<\dots<i_d$, let 
$$g_S=g_{0i_1}\circ g_{i_1i_2} \circ\dots\circ g_{i_{d-1}i_d}.$$
Then, if $\lambda\in\S_n$ is the reverse-layered permutation with descent set $S$, Lemma~\ref{lem:g} implies that
\begin{equation}\label{eq:g-des}
\des_\lambda(T)=\des_{12\dots n}(g_S(T))+d \quad\text{and}\quad \plat(T)=\plat(g_S(T))
\end{equation}
for all $T\in\SYT(k^n)$. See Figure~\ref{fig:gS} for an example of this bijection.

\begin{figure}[htb]
\centering
 \begin{tikzpicture}[scale=.5]
 
\fill[red!30] (0,0) rectangle (3,2);
\fill[blue!30] (0,3) rectangle (3,8);
\draw (0,0) grid (3,8);
\node at (.5,7.5) {$1$}; \node at (1.5,7.5) {$5$}; \node at (2.5,7.5) {$10$}; 
\node at (.5,6.5) {$2$}; \node at (1.5,6.5) {$8$}; \node at (2.5,6.5) {$15$}; 
\node at (.5,5.5) {$3$}; \node at (1.5,5.5) {$11$}; \node at (2.5,5.5) {$17$}; 
\node at (.5,4.5) {$4$}; \node at (1.5,4.5) {$13$}; \node at (2.5,4.5) {$18$}; 
\node at (.5,3.5) {$6$}; \node at (1.5,3.5) {$14$}; \node at (2.5,3.5) {$20$}; 
\node at (.5,2.5) {$7$}; \node at (1.5,2.5) {$16$}; \node at (2.5,2.5) {$22$};  
\node at (.5,1.5) {$9$}; \node at (1.5,1.5) {$19$}; \node at (2.5,1.5) {$23$};  
\node at (.5,.5) {$12$}; \node at (1.5,.5) {$21$}; \node at (2.5,.5) {$24$};  
\node[left] at (0,4) {$T=$};
\draw[->] (4,4)--node[above] {$g_{56}$} (5.5,4);
 
 \begin{scope}[shift={(6.5,0)}]
 \fill[red!30] (0,0) rectangle (3,3);
  \fill[blue!30] (0,6) rectangle (3,8);
\draw (0,0) grid (3,8);
\node at (.5,7.5) {$1$}; \node at (1.5,7.5) {$5$}; \node at (2.5,7.5) {$10$}; 
\node at (.5,6.5) {$2$}; \node at (1.5,6.5) {$8$}; \node at (2.5,6.5) {$15$}; 
\node at (.5,5.5) {$3$}; \node at (1.5,5.5) {$11$}; \node at (2.5,5.5) {$18$}; 
\node at (.5,4.5) {$4$}; \node at (1.5,4.5) {$13$}; \node at (2.5,4.5) {$19$}; 
\node at (.5,3.5) {$6$}; \node at (1.5,3.5) {$14$}; \node at (2.5,3.5) {$21$}; 
\node at (.5,2.5) {$7$}; \node at (1.5,2.5) {$16$}; \node at (2.5,2.5) {$22$};  
\node at (.5,1.5) {$9$}; \node at (1.5,1.5) {$17$}; \node at (2.5,1.5) {$23$};  
\node at (.5,.5) {$12$}; \node at (1.5,.5) {$20$}; \node at (2.5,.5) {$24$};  
\draw[->] (4,4)--node[above] {$g_{25}$} (5.5,4);
\end{scope}

\begin{scope}[shift={(13,0)}]
 \fill[red!30] (0,0) rectangle (3,6);
\draw (0,0) grid (3,8);
\node at (.5,7.5) {$1$}; \node at (1.5,7.5) {$5$}; \node at (2.5,7.5) {$9$}; 
\node at (.5,6.5) {$2$}; \node at (1.5,6.5) {$7$}; \node at (2.5,6.5) {$17$}; 
\node at (.5,5.5) {$3$}; \node at (1.5,5.5) {$11$}; \node at (2.5,5.5) {$18$}; 
\node at (.5,4.5) {$4$}; \node at (1.5,4.5) {$13$}; \node at (2.5,4.5) {$19$}; 
\node at (.5,3.5) {$6$}; \node at (1.5,3.5) {$14$}; \node at (2.5,3.5) {$21$}; 
\node at (.5,2.5) {$8$}; \node at (1.5,2.5) {$15$}; \node at (2.5,2.5) {$22$};  
\node at (.5,1.5) {$10$}; \node at (1.5,1.5) {$16$}; \node at (2.5,1.5) {$23$};  
\node at (.5,.5) {$12$}; \node at (1.5,.5) {$20$}; \node at (2.5,.5) {$24$};  
\draw[->] (4,4)--node[above] {$g_{02}$} (5.5,4);
\end{scope}

\begin{scope}[shift={(19.5,0)}]
\draw (0,0) grid (3,8);
\node at (.5,7.5) {$1$}; \node at (1.5,7.5) {$5$}; \node at (2.5,7.5) {$9$}; 
\node at (.5,6.5) {$2$}; \node at (1.5,6.5) {$7$}; \node at (2.5,6.5) {$17$}; 
\node at (.5,5.5) {$3$}; \node at (1.5,5.5) {$11$}; \node at (2.5,5.5) {$18$}; 
\node at (.5,4.5) {$4$}; \node at (1.5,4.5) {$13$}; \node at (2.5,4.5) {$19$}; 
\node at (.5,3.5) {$6$}; \node at (1.5,3.5) {$14$}; \node at (2.5,3.5) {$21$}; 
\node at (.5,2.5) {$8$}; \node at (1.5,2.5) {$15$}; \node at (2.5,2.5) {$22$};  
\node at (.5,1.5) {$10$}; \node at (1.5,1.5) {$16$}; \node at (2.5,1.5) {$23$};  
\node at (.5,.5) {$12$}; \node at (1.5,.5) {$20$}; \node at (2.5,.5) {$24$};  
\node[right] at (3,4) {$=g_{\{2,5,6\}}(T)$};
\end{scope}
 \end{tikzpicture}
 \caption{The bijection $g_S=g_{02}\circ g_{25}\circ g_{56}$ for $S=\{2,5,6\}$. Letting $\lambda=78456312$, we have $\des_\lambda(T)=10=\des_{12\dots 8}(g_S(T))+3$ and $\plat(T)=\plat(g_{S}(T))=0$.}
\label{fig:gS}
\end{figure}
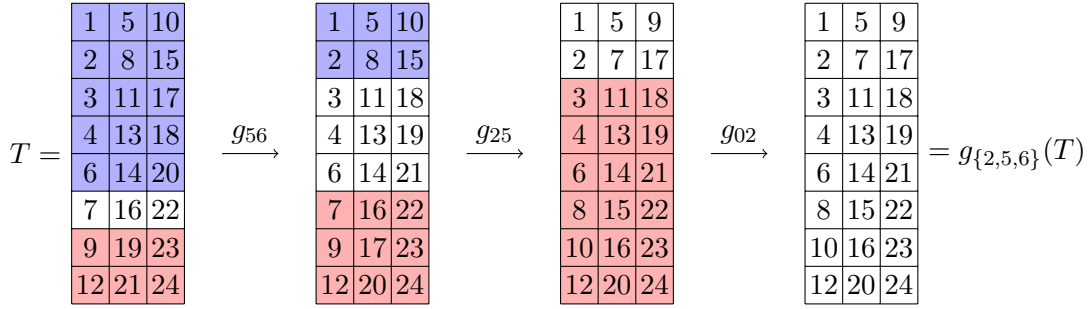

Finally, for any $\sigma\in\S_n$, let $S=\Des(\sigma)$ and define $\bij_\sigma=g_S\circ f_\sigma$. Combining equations~\eqref{eq:f-des} and~\eqref{eq:g-des}, we see that $\bij_\sigma$ satisfies the property stated in equation~\eqref{eq:bij_property}. See Figure~\ref{fig:bij} for a complete example of this bijection.

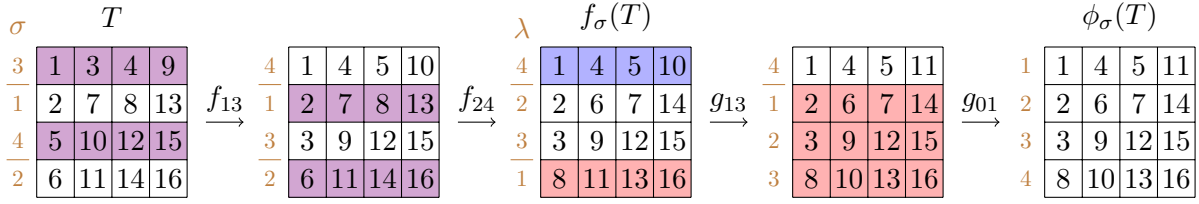
\begin{figure}[htb]
\centering
 \begin{tikzpicture}[scale=.5]
\fill[violet!35] (0,3) rectangle (4,4);
\fill[violet!35] (0,1) rectangle (4,2);
\draw (0,0) grid (4,4);
\node at (.5,3.5) {$1$}; \node at (1.5,3.5) {$3$}; \node at (2.5,3.5) {$4$}; \node at (3.5,3.5) {$9$}; 
\node at (.5,2.5) {$2$}; \node at (1.5,2.5) {$7$}; \node at (2.5,2.5) {$8$};  \node at (3.5,2.5) {$13$}; 
\node at (.5,1.5) {$5$}; \node at (1.5,1.5) {$10$}; \node at (2.5,1.5) {$12$};  \node at (3.5,1.5) {$15$}; 
\node at (.5,.5) {$6$}; \node at (1.5,.5) {$11$}; \node at (2.5,.5) {$14$};  \node at (3.5,.5) {$16$}; 
\node[brown] at (-.5,4.5) {$\sigma$}; \node[brown,scale=.8] at (-.5,3.5) {$3$}; \desbar{(-.5,3)} \node[brown,scale=.8] at (-.5,2.5) {$1$}; \node[brown,scale=.8] at (-.5,1.5) {$4$}; \desbar{(-.5,1)} \node[brown,scale=.8] at (-.5,.5) {$2$};
\node at (2,4.8) {$T$};
\draw[->] (4.5,2)--node[above] {$f_{13}$} (5.5,2);
 
 \begin{scope}[shift={(6.7,0)}]
\fill[violet!35] (0,2) rectangle (4,3);
\fill[violet!35] (0,0) rectangle (4,1);
\draw (0,0) grid (4,4);
\node at (.5,3.5) {$1$}; \node at (1.5,3.5) {$4$}; \node at (2.5,3.5) {$5$}; \node at (3.5,3.5) {$10$}; 
\node at (.5,2.5) {$2$}; \node at (1.5,2.5) {$7$}; \node at (2.5,2.5) {$8$};  \node at (3.5,2.5) {$13$}; 
\node at (.5,1.5) {$3$}; \node at (1.5,1.5) {$9$}; \node at (2.5,1.5) {$12$};  \node at (3.5,1.5) {$15$}; 
\node at (.5,.5) {$6$}; \node at (1.5,.5) {$11$}; \node at (2.5,.5) {$14$};  \node at (3.5,.5) {$16$}; 
\node[brown,scale=.8] at (-.5,3.5) {$4$}; \desbar{(-.5,3)} \node[brown,scale=.8] at (-.5,2.5) {$1$}; \node[brown,scale=.8] at (-.5,1.5) {$3$}; \desbar{(-.5,1)} \node[brown,scale=.8] at (-.5,.5) {$2$};
\draw[->] (4.5,2)--node[above] {$f_{24}$} (5.5,2);
\end{scope}

\begin{scope}[shift={(13.4,0)}]
 \fill[red!30] (0,0) rectangle (4,1);
  \fill[blue!30] (0,3) rectangle (4,4);
\draw (0,0) grid (4,4);
\node at (.5,3.5) {$1$}; \node at (1.5,3.5) {$4$}; \node at (2.5,3.5) {$5$}; \node at (3.5,3.5) {$10$}; 
\node at (.5,2.5) {$2$}; \node at (1.5,2.5) {$6$}; \node at (2.5,2.5) {$7$};  \node at (3.5,2.5) {$14$}; 
\node at (.5,1.5) {$3$}; \node at (1.5,1.5) {$9$}; \node at (2.5,1.5) {$12$};  \node at (3.5,1.5) {$15$}; 
\node at (.5,.5) {$8$}; \node at (1.5,.5) {$11$}; \node at (2.5,.5) {$13$};  \node at (3.5,.5) {$16$}; 
\node[brown] at (-.5,4.5) {$\lambda$}; \node[brown,scale=.8] at (-.5,3.5) {$4$}; \desbar{(-.5,3)} \node[brown,scale=.8] at (-.5,2.5) {$2$}; \node[brown,scale=.8] at (-.5,1.5) {$3$}; \desbar{(-.5,1)} \node[brown,scale=.8] at (-.5,.5) {$1$};
\node at (2,4.8) {$f_\sigma(T)$};
\draw[->] (4.5,2)--node[above] {$g_{13}$} (5.5,2);
\end{scope}

\begin{scope}[shift={(20.1,0)}]
 \fill[red!30] (0,0) rectangle (4,3);
 \draw (0,0) grid (4,4);
\node at (.5,3.5) {$1$}; \node at (1.5,3.5) {$4$}; \node at (2.5,3.5) {$5$}; \node at (3.5,3.5) {$11$}; 
\node at (.5,2.5) {$2$}; \node at (1.5,2.5) {$6$}; \node at (2.5,2.5) {$7$};  \node at (3.5,2.5) {$14$}; 
\node at (.5,1.5) {$3$}; \node at (1.5,1.5) {$9$}; \node at (2.5,1.5) {$12$};  \node at (3.5,1.5) {$15$}; 
\node at (.5,.5) {$8$}; \node at (1.5,.5) {$10$}; \node at (2.5,.5) {$13$};  \node at (3.5,.5) {$16$}; 
\node[brown,scale=.8] at (-.5,3.5) {$4$}; \desbar{(-.5,3)} \node[brown,scale=.8] at (-.5,2.5) {$1$}; \node[brown,scale=.8] at (-.5,1.5) {$2$}; \node[brown,scale=.8] at (-.5,.5) {$3$};
\draw[->] (4.5,2)--node[above] {$g_{01}$} (5.5,2);
\end{scope}

\begin{scope}[shift={(26.8,0)}]
 \draw (0,0) grid (4,4);
\node at (.5,3.5) {$1$}; \node at (1.5,3.5) {$4$}; \node at (2.5,3.5) {$5$}; \node at (3.5,3.5) {$11$}; 
\node at (.5,2.5) {$2$}; \node at (1.5,2.5) {$6$}; \node at (2.5,2.5) {$7$};  \node at (3.5,2.5) {$14$}; 
\node at (.5,1.5) {$3$}; \node at (1.5,1.5) {$9$}; \node at (2.5,1.5) {$12$};  \node at (3.5,1.5) {$15$}; 
\node at (.5,.5) {$8$}; \node at (1.5,.5) {$10$}; \node at (2.5,.5) {$13$};  \node at (3.5,.5) {$16$}; 
\node[brown,scale=.8] at (-.5,3.5) {$1$}; \node[brown,scale=.8] at (-.5,2.5) {$2$}; \node[brown,scale=.8] at (-.5,1.5) {$3$}; \node[brown,scale=.8] at (-.5,.5) {$4$};
\node at (2,4.8) {$\bij_\sigma(T)$};
\end{scope}
 \end{tikzpicture}
 \caption{The bijection $\bij_\sigma=g_{\Des(\sigma)}\circ f_\sigma$ for $\sigma=3142$, which has $\Des(\sigma)=\{1,3\}$. In this example,
  $\des_\sigma(T)=6=\des_{1234}(\bij_\sigma(T))+2$ and $\plat(T)=\plat(\bij_\sigma(T))=2$.}
\label{fig:bij}
\end{figure}

\section{Symmetries}

A consequence of Theorem~\ref{thm:main} is that the distribution of the number of descents on canon permutations is symmetric. For $k=2$, this fact was already noted and proved bijectively in~\cite{elizalde_descents_2023}.

\begin{corollary}\label{cor:symmetry}
For every $0\le h\le k(n-1)$,
$$|\{\pi\in\C^k_n:\des(\pi)=h\}|=|\{\pi\in\C^k_n:\des(\pi)=k(n-1)-h\}|.$$
\end{corollary}

\begin{proof}
It is well known that the Eulerian polynomials are palindromic, that is, $A_n(t)=t^{n-1}A_n(1/t)$, noting that
$\des(\sigma)=n-1-\des(\sigma_n\sigma_{n-1}\dots\sigma_1)$ for all $\sigma=\sigma_1\dots\sigma_n\in\S_n$.

A much less obvious fact is that the generalized Narayana polynomials are palindromic as well. Indeed, it is shown by Sulanke \cite[Cor.~1]{sulanke_generalizing_2004} that 
\begin{equation}\label{eq:symN} N(n,k,h)=N(n,k,(k-1)(n-1)-h) \end{equation}
for all $0\le h\le (k-1)(n-1)$, or equivalently,
$N_{n,k}(t,1)=t^{(k-1)(n-1)}N_{n,k}(1/t,1)$.

Using these symmetries together with Theorem~\ref{thm:main},
$$\poly^k_n(t,1)=A_n(t)\,N_{n,k}(t,1)=t^{k(n-1)}A_n(1/t)N_{n,k}(1/t,1)=t^{k(n-1)}\poly^k_n(1/t,1).$$
Taking the coefficient of $t^h$ on both sides, we obtain the stated equality.
\end{proof}

The above proof of Corollary~\ref{cor:symmetry} is not bijective, due to the fact that the symmetry of the generalized Narayana numbers, stated in equation~\eqref{eq:symN}, is not proved bijectively in~\cite{sulanke_generalizing_2004}. A bijective proof of equation~\eqref{eq:symN} will be given in~\cite{elizalde_generalized_nodate}.

\bibliographystyle{plain}
\bibliography{symmetry_des_rectangular_SYT}

\begin{thebibliography}{1}

\bibitem{archer_pattern_2019}
Kassie Archer, Adam Gregory, Bryan Pennington, and Stephanie Slayden.
\newblock Pattern restricted quasi-{Stirling} permutations.
\newblock {\em The Australasian Journal of Combinatorics}, 74:389--407, 2019.

\bibitem{bernardi_deformations_2018}
Olivier Bernardi.
\newblock Deformations of the braid arrangement and trees.
\newblock {\em Advances in Mathematics}, 335:466--518, September 2018.

\bibitem{elizalde_generalized_nodate}
Sergi Elizalde.
\newblock A generalized {Lalanne}--{Kreweras} involution for rectangular
  tableaux.
\newblock In preparation.

\bibitem{elizalde_descents_2021}
Sergi Elizalde.
\newblock Descents on quasi-{Stirling} permutations.
\newblock {\em Journal of Combinatorial Theory. Series A}, 180:Paper No.
  105429, 35, 2021.

\bibitem{elizalde_descents_2023}
Sergi Elizalde.
\newblock Descents on nonnesting multipermutations.
\newblock {\em European Journal of Combinatorics}, page 103846, October 2023.

\bibitem{gessel_stirling_1978}
Ira~M. Gessel and Richard~P. Stanley.
\newblock Stirling polynomials.
\newblock {\em Journal of Combinatorial Theory. Series A}, 24(1):24--33, 1978.

\bibitem{macmahon_combinatory_1960}
Percy~A. MacMahon.
\newblock {\em Combinatory analysis}.
\newblock Chelsea Publishing Co., New York, 1960.

\bibitem{stanley_enumerative_2012}
Richard~P. Stanley.
\newblock {\em Enumerative combinatorics. {Vol}. 1}, volume~49 of {\em
  Cambridge {Studies} in {Advanced} {Mathematics}}.
\newblock Cambridge University Press, Cambridge, second edition, 2012.

\bibitem{sulanke_generalizing_2004}
Robert~A. Sulanke.
\newblock Generalizing {Narayana} and {Schröder} {Numbers} to {Higher}
  {Dimensions}.
\newblock {\em The Electronic Journal of Combinatorics}, pages R54--R54, August
  2004.

\end{thebibliography}

\end{document}